\documentclass{amsart}
\pdfoutput=1

\usepackage[utf8]{inputenc}
\usepackage[T1]{fontenc}

\usepackage{amsfonts,amssymb,amsthm,amsmath}
\usepackage{algorithm}
\usepackage{algpseudocode}
\usepackage{graphbox,graphicx}
\usepackage{subcaption}
\usepackage{multirow}
\usepackage{booktabs}

\usepackage{hyperref}

\usepackage[backend=biber,maxnames=50,giveninits=true,style=numeric-comp]{biblatex}
\calclayout

\graphicspath{{figs/}}

\newtheorem{theorem}{Theorem}[section]

\newtheorem{lemma}[theorem]{Lemma}
\newtheorem{definition}[theorem]{Definition}
\newtheorem{corollary}[theorem]{Corollary}

\newtheorem{remark}[theorem]{Remark}

\DeclareMathOperator*{\argmin}{arg\,min}

\makeatletter
\@namedef{subjclassname@2020}{%
	\textup{2020} Mathematics Subject Classification}
\makeatother

\title{Efficient Random Walks on Riemannian Manifolds}

\author{Simon Schwarz}
\address{Institute for Mathematical Stochastics\\ 
	University of Göttingen\\
	Goldschmidtstraße 7\\ 37077 Göttingen\\
	Germany}
\email{simon.schwarz@uni-goettingen.de}

\author{Michael Herrmann}
\address{Institute for Partial Differential Equations\\ 
Technische Universität Braunschweig\\
Universitäts\-platz 2\\
38106 Braunschweig\\
Germany}
\email{michael.herrmann@tu-braunschweig.de}

\author{Anja Sturm}
\address{Institute for Mathematical Stochastics\\ 
University of Göttingen\\
Goldschmidtstraße 7\\
37077 Göttingen\\
Germany}
\email{anja.sturm@mathematik.uni-goettingen.de}

\author{Max Wardetzky}
\address{Institute for Numerical and Applied Mathematics\\ 
University of Göttingen\\
Lotzestr. 16-18\\
37083 Göttingen\\
Germany}
\email{wardetzky@math.uni-goettingen.de}

\thanks{We acknowledge support by the Collaborative Research Center 1456 funded by the Deutsche Forschungsgemeinschaft (DFG, German Research Foundation). We thank Pit Neumann for pointing out an error in a previous version of the manuscript and the anonymous referees for improving the quality of our manuscript.}

\begin{document}



\begin{abstract}
According to a version of Donsker's theorem, geodesic random walks on Riemannian manifolds converge to the respective Brownian motion. From a computational perspective, however, evaluating geodesics can be quite costly. We therefore introduce approximate geodesic random walks based on the concept of retractions. We show that these approximate walks converge in distribution to the correct Brownian motion as long as the geodesic equation is approximated up to second order. As a result we obtain an efficient algorithm for sampling Brownian motion on compact Riemannian manifolds.
\end{abstract}
\vspace*{-1cm}
\maketitle

\section{Introduction}
\label{intro}
Probabilistic models in continuous time with geometric constraints lead to Brownian motion and stochastic differential equations (SDEs) on Riemannian manifolds $(M,g)$. The theory of Brownian motion and SDEs on manifolds has been extensively studied (see, e.g., ~\cite{emery1989,hsu2002}) -- and many results from the Euclidean setting can be generalized to manifold-valued SDEs. 
One example is Donsker's theorem. Here, one considers a geodesic random walk for which the update step at the current location $x_i$ is obtained by uniformly sampling a unit tangent vector $\bar{v} \sim \operatorname{Unif}( S_{x_i})$ with $S_x := \{u\in T_{x}M : \lVert u\rVert_g = 1\}$ and following a geodesic path by setting $x_{i+1} := \operatorname{Exp}_{x_i} (\varepsilon v)$ for some small enough parameter $\varepsilon$, where $v:=\sqrt{m} \bar{v}$, $m$ denotes the dimension of the manifold, and $\operatorname{Exp}_{x}:T_xM \to M$ is the exponential map. We will discuss the reason for scaling by $\sqrt{m}$ below. J{\o}rgensen proved that on complete Riemannian manifolds that satisfy mild assumptions (which in particular hold for compact manifolds), the scaling limit of such a geodesic random walk is the Brownian motion on $(M,g)$, see~\cite{jorgensen1975}. This result has recently been generalized to the setting of Finsler manifolds, see~\cite{ma2021}.

In principle, it is thus possible to simulate Brownian motion on Riemannian manifolds by computing geodesic random walks -- analogously to how one would use Donsker's theorem in the Euclidean case. Unfortunately, this procedure is not computationally efficient on manifolds since the solution of a non-linear second-order geodesic ODE has to be computed in every step. With an eye on applications, we are interested in efficient and simple, yet convergent, methods in order to approximate Brownian motion on Riemannian manifolds. 

We introduce a class of algorithms that yield approximate geodesic random walks by making use of \emph{retractions} that approximate the exponential map and have been introduced in the context of numerical optimization on manifolds, see~\cite{absil2008,absil2012,adler2002,bonnabel2013}. We show that retractions which approximate the geodesic ODE up to second order converge to Brownian motion in the Skorokhod topology, see Section~\ref{sec:convergence}. We consider two prevalent scenarios, where random walks based on such second-order retractions can be computed efficiently: The case of parameterized manifolds and the case of implicitly defined manifolds, see Section~\ref{sec:retraction-walks}. Our approach generalizes the setting of approximate geodesic walks on manifolds with positive sectional curvatures presented in~\cite{mangoubi2018}. Moreover, our method works for arbitrary dimensions, and its cost only depends on the evaluation of the respective retraction.

A retraction in our sense can be understood as both, an approximation of the geodesic ODE, or, equivalently, as a 2-jet in the sense of~\cite{armstrong2018coordinatefree}, which allows for treating coordinate-free SDEs. Armstrong and King developed a 2-jet scheme for simulating SDEs on manifolds in \cite{armstrong2020curved}, which uses an unbounded stepsize in each iteration. Different from our re\-trac\-tion-based approach, which is based on \emph{bounded} stepsizes, the case of unbounded stepsizes hampers efficiency (for the case of local charts) or might even prohibit an implementation (for the case of implicitly defined sub\-ma\-ni\-folds when using a step-and-project approach). Moreover, our random-walk method is based on a Donsker-type theorem, while their approach improves Euler--Maruyama approximations -- resulting in a different type of convergence in terms of the respective topologies.

Outside the setting of 2-jets, one typically considers SDEs on embedded manifolds $M \subset \mathbb{R}^n$ and first solves them in ambient Euclidean space, followed by projecting back to $M$ using Lagrange multipliers, see, e.g.,~\cite{leimkuhler2015,averina2019}.  Projection-based approaches, however, do not immediately extend to the case of parameterized manifolds. Within our unified framework based on retractions, we cover both, the case of embedded manifolds and the case of parameterized manifolds, see Sections~\ref{ssec:example1} \&~\ref{ssec:example2}. 
In particular, our approach leads to an efficient and convergent numerical treatment of Brownian motion on Riemannian manifolds (including drift -- see Remark \ref{rem:diffusion}).

A different scenario where geodesic random walks are commonly used is the problem of sampling the volume measure of polytopes (defined as the convex hull resulting from  certain linear constraints). Originally, Euclidean random walk methods had been considered in this context, see~\cite{vempala2005}. Although the Euclidean perspective is natural in this situation, small stepsizes are required in order to reach points near the boundary, which hampers efficiency in higher dimensions. To circumvent this issue, Lee and Vempala \cite{lee2017polytopes} introduced a Riemannian metric based on the Hessian of a log-barrier -- to the effect of re-scaling space as the walk approaches the boundary. In their implementation, they used collocation schemes for solving the requisite ODEs in order to simulate geodesic random walks.
We leave for future research the question whether such approaches can be reformulated in terms of retractions. 

Furthermore, the problem of sampling on manifolds is commonly addressed by Markov chain Monte Carlo (MCMC) methods, see, e.g.,~\cite{brubaker2012,byrne2013,cantarella2016,yanush2019mcmc}.
Different from our focus, MCMC methods are typically concerned with obtaining samples from a given probability distribution on $M$. Our algorithm, however, yields approximate sample \emph{paths} of a Brownian motion on a manifold.
We are not aware of any results concerned with proving that an MCMC method samples the correct dynamics of a given SDE.
Although not the main focus of our exposition, we show that our method also correctly recovers the stationary measure of a geodesic random walk in the limit of stepsize tending to zero, see Theorem~\ref{thm:inv-measure}.

\section{Retraction-based random walks}\label{sec:retraction-walks}
Throughout this exposition, we consider $m$-dimensional compact and orientable Riemannian manifolds $(M,g)$ without boundary. We use \emph{retractions} to approximate the exponential map on such manifolds.

\begin{definition}\label{def:retraction}
	Let $\operatorname{Ret}: TM\xrightarrow{} M$ be a smooth map, and denote the restriction to the tangent space $T_xM$ by $\operatorname{Ret}_x$ for any $x\in M$. $\operatorname{Ret}$ is a \emph{retraction} if the following two conditions are satisfied for all $x\in M$ and all $v\in T_xM$:
	\begin{enumerate}
		\item $\operatorname{Ret}_x(0) = x$, where $0$ is the zero element in $T_xM$ and
		\item $\frac{d}{d\tau}\operatorname{Ret}_x (\tau v)\big\lvert_{\tau = 0} = v$ (where we identify $T_0T_xM \simeq T_xM$). %$\textnormal{d}\operatorname{Ret}_x\lvert_{0} = \operatorname{id}_{T_xM}$
	\end{enumerate}
\end{definition}
A retraction is a \emph{second-order retraction} if it additionally satisfies that for all $x \in M$ and for all $v\in T_xM$ one has that
\begin{equation}\label{eq:second-order}
\frac{D}{d\tau}\left(\frac{d}{d\tau}\operatorname{Ret}_x (\tau v)\right)\Bigg\lvert_{\tau =0}  =\frac{D}{d\tau}\left(\frac{d}{d\tau}\operatorname{Exp}_x (\tau v)\right)\Bigg\lvert_{\tau=0} = 0 \ ,
\end{equation}
where $\frac{D}{d\tau}(\frac{d}{d\tau} \gamma(\tau))$ denotes covariant differentiation of the tangent vector field $\dot \gamma (\tau) = \frac{d}{d\tau} \gamma(\tau)$ along the curve $\gamma$ (following standard notation, see, e.g., \cite{docarmo1992geometry}).
If $M$ is a submanifold of a Euclidean space, Equation \eqref{eq:second-order} is equivalent to
\begin{equation*}
\frac{d^2}{d\tau^2} \operatorname{Ret}_x (\tau v)\big\lvert_{\tau=0} \in \mathcal{N}_xM\text{,}
\end{equation*}
where $\mathcal{N}_xM$ denotes the normal bundle of $M$. Consequently, for the case of submanifolds, one has 
\begin{equation*}
\operatorname{Ret}_x (\tau v) = \operatorname{Exp}_x (\tau v) + \mathcal{O}(\tau^3)
\end{equation*}
for all $x\in M$ and $v\in T_xM$ as $\tau\rightarrow 0$.

Clearly, the exponential map is itself a retraction. The main benefit of retractions is, however, that they can serve as computationally efficient \emph{approximations} of the exponential map. 
Using retractions yields Algorithm \ref{alg:retraction} for simulating random walks on a Riemannian manifold.

\begin{algorithm}
	\caption{Retraction-based random walk}\label{alg:retraction}
	\begin{algorithmic}[1]
		\State \textbf{Input:} Retraction $\operatorname{Ret}$; iterations $N$; stepsize $\varepsilon$; dimension $m$; initial position $x_0$; set $x_0^\varepsilon := x_0$
		\For{$1 \leq i\leq N$}
		\State Sample $\bar{v} \in S_{x_{i-1}^\varepsilon}:=\{u\in T_{x_{i-1}^\varepsilon}M : \lVert u\rVert_g = 1\}$ uniformly wrt.~the Riemannian metric $g$
		\State Set $v:=\sqrt{m}\bar{v}$
		\State $x_i^\varepsilon := \operatorname{Ret}_{x_{i-1}^\varepsilon} (\varepsilon  v)$
		\EndFor
		\State \textbf{Return:} $(x_j^\varepsilon)_{0\leq j\leq N}$
	\end{algorithmic}
\end{algorithm}
Notice that the uniform sampling in Step 3 of the above algorithm is performed with respect to the Riemannian metric restricted to the \emph{unit sphere} in the attendant tangent space. We provide details for this step in Sections~\ref{ssec:example1} and~\ref{ssec:example2} below. Notice furthermore that $\varepsilon$ takes the role of a \emph{spatial} stepsize, while the total physical \emph{time} %resulting from  
simulated by Algorithm \ref{alg:retraction} is given by $\varepsilon^2N$; see also the time rescaling in Theorem \ref{thm:general-convergence} below.
Step 4 of the algorithm is necessary since uniform sampling on the unit sphere in a space of dimension $m$ results in a covariance matrix $\frac{1}{m}\operatorname{Id}$, while convergence to Brownian motion requires the identity as covariance matrix. The normalizing  factor $\sqrt{m}$ in Step~4 precisely ensures the latter. Indeed, without the renormalization in Step 4 the random walk constructed in Algorithm~\ref{alg:retraction} converges to a time-changed Brownian motion.

The next theorem shows convergence of the retraction-based random walk resulting from Algorithm~\ref{alg:retraction} and provides an explicit expression for the generator of the resulting limit process: 

\begin{theorem}\label{thm:general-convergence}
	Consider the sequence of random variables $(X_i^\varepsilon)_{i\in\mathbb{N}}$ constructed in Algorithm~\ref{alg:retraction} (with $N=\infty$). The continuous-time process $X^\varepsilon := \big(X_{\lfloor \varepsilon^{-2}t\rfloor}^\varepsilon\big)_{t\geq 0}$ converges in distribution to the stochastic process with generator
	\begin{equation}\label{eq:generator}
	(Lf)(x) = \frac12 (\Delta_gf)(x) 
	+ \frac{m}{2\omega_m}\int_{S_x} df\big\lvert_x \left(\frac{D}{d\tau} \frac{d}{d\tau} \operatorname{Ret}_x (\tau \bar{v})\right)\Bigg\lvert_{\tau = 0} \text{d}\bar{v}
	\end{equation}
	in the Skorokhod topology (see Section \ref{sec:convergence} for details). Here $S_x$ denotes the unit sphere in $T_x M$, $\omega_m$ is the volume of this sphere, and $\Delta_g$ is the Laplace--Beltrami operator on $(M,g)$.
\end{theorem}

We defer the proof to Section~\ref{sec:convergence}. Notice that Theorem \ref{thm:general-convergence} immediately implies a necessary and sufficient condition for a retraction-based random walk to converge to Brownian motion on a Riemannian manifold $(M,g)$:
\begin{corollary}
	Using the same notation as in Theroem~\ref{thm:general-convergence}, the continuous-time process $X^\varepsilon$ converges in distribution to the Brownian motion on $M$ if and only if the Laplacian (also known as the \emph{tension}) of $\operatorname{Ret}_x: T_xM\xrightarrow{} M$ vanishes at $0 \in T_{x}M$ for every $x\in M$. In particular, this holds for second-order retractions.
\end{corollary}

\begin{proof}
	By definition, the Laplacian of $\operatorname{Ret}_x: T_xM\xrightarrow{} M$ is defined as
	\begin{equation*}
	\Delta \operatorname{Ret}_x := \operatorname{trace} \nabla d  \operatorname{Ret}_x \ ,
	\end{equation*}
	where $\nabla d  \operatorname{Ret}_x$ is the Hessian of the $C^\infty$-map $\operatorname{Ret}_x$.
	Clearly, the Laplacian of $\operatorname{Ret}_x$ vanishes at $0 \in T_{x}M$ if and only if
	\begin{equation*}
	0= \int_{S_x}\frac{D}{d\tau} \frac{d}{d\tau} \operatorname{Ret}_x (\tau \bar{v})\Bigg\lvert_{\tau = 0} \text{d}\bar{v} \ . 
	\end{equation*}	
	Applying Theorem~\ref{thm:general-convergence} proves the first claim. If $\operatorname{Ret}_x$ is a second-order retraction, the second claim immediately follows since $\frac{D}{d\tau} \frac{d}{d\tau} \operatorname{Ret}_x (\tau v)\big\lvert_{\tau = 0}=0$ for all $v\in T_xM$.
\end{proof}

Thus, random walks generated by second-order retractions converge to the Brownian motion on $(M,g)$. 

\begin{remark}\label{rem:diffusion}
	Drift can seamlessly be incorporated into our approach. Indeed, when modeling drift by a vector field $X$, then Step 4 of Algorithm~\ref{alg:retraction} needs to be altered according to $v:=\sqrt{m}\bar{v} + \varepsilon X\lvert_{x_{i-1}^\varepsilon}$. The stochastic process resulting from this altered version of Algorithm~\ref{alg:retraction} converges to the process with generator $Lf+df(X)$, where $Lf$ is the generator given by~\eqref{eq:generator}. The proof of this fact follows the proof of Theorem~\ref{thm:general-convergence} with obvious modifications.
\end{remark}

In the next two subsections we consider two concrete examples of second-order retractions that are \emph{computationally efficient}. 

\subsection{Retractions based on local parameterizations}\label{ssec:example1}
Our first example of computationally efficient retractions is based on local  parameterizations of compact $m$-dimensional Riemannian manifolds $(M,g)$. % \emph{isometrically} embedded into some $\mathbb{R}^n$. 
Consider an atlas consisting of finitely many charts $\{(U_i, \phi_i)\}_{i\in I}$, where every $U_i \subset \mathbb{R}^m$ is open, $\phi_i:U_i \to M$ 
%\subset \mathbb{R}^n$ 
is a diffeomorphism onto its image in $M$, and where the relatively open sets $\{\phi_i(U_i)\}$ cover $M$. In concrete applications, such parameterizations often arise naturally, e.g., for isometrically embedded manifolds $M\subset\mathbb{R}^{n}$. 

For a given $x\in M$, $v\in T_xM$, and $(U_i,\phi_i)$ with $x\in \phi_i(U_i)$, let $\tilde{x} = \phi_i^{-1} (x)$ and $\tilde{v} = (d\phi_i^{-1})\lvert_x v$. 
We assume that for some fixed $\varepsilon > 0$ and any $x\in M$, there exists a chart $(U_k, \phi_k)$ with $\tilde x\in U_{k}$ such that 
$$\tilde{x}+\tilde{v}-\frac{1}{2}\Gamma_{ij}\tilde{v}^i\tilde{v}^j \in U_k$$
for all $\tilde v =(d\phi_k^{-1})\lvert_x (\varepsilon \sqrt{m}\,\bar{v})$ with $\bar{v}\in S_x$, where the $(\Gamma_{ij})$ denote the Christoffel symbols of the respective chart.
This assumption is readily satisfied for any compact $M$, provided that $\varepsilon$ is chosen small enough.
In order to find such a sufficiently small $\varepsilon$ in Algorithm \ref{alg:retraction}, we propose to restart the algorithm with stepsize $\frac{\varepsilon}{2}$ if the condition is violated. %If multiple charts satisfy this assumption in a point $x\in M$, we choose one in order to get a well defined map $x\mapsto (U_k, \phi_k)$.
Then we define the retraction
\begin{equation}\label{eq:ret-param}
\operatorname{p-Ret}_x (v) := \phi_k\left(\tilde{x}+\tilde{v}-\frac{1}{2}\Gamma_{ij}\tilde{v}^i\tilde{v}^j\right)\text{.}
\end{equation} 
Notice that Christoffel symbols can be computed numerically. For low-dimensional manifolds, symbolic differentiation accelerates their computation.

\begin{lemma}
	The retraction defined by Equation~\eqref{eq:ret-param} is of second order.
\end{lemma}
\begin{proof}
	Let $(\gamma(\tau))_{\tau\geq 0}$ denote the geodesic satisfying $\gamma (0) = x$ and $\dot{\gamma}(0) = v$. In the parameter domain, $\gamma$ satisfies the ordinary differential equation
	\begin{equation}\label{eq:geodesic-ode}
	\ddot{\gamma}^k + \Gamma_{ij}^k \dot{\gamma}^i\dot{\gamma}^j = 0 \ .
	\end{equation}
	The Taylor expansion of $\gamma$ reads
	\begin{equation*}
	\begin{split}
	\gamma^k (\tau) &= \gamma^k (0) + \tau \Dot{\gamma}^k (0) + \frac{\tau^2}{2} \Ddot{\gamma}^k (0)  + \mathcal{O}(\tau^3) \\
	&= \gamma^k (0) + \tau\Dot{\gamma}^k (0) - \frac{\tau^2}{2}\Gamma_{ij}^k \Dot{\gamma}^i(0)\Dot{\gamma}^j(0) + \mathcal{O}(\tau^3) \\
	&= \operatorname{p-Ret}^k_x(v) +\mathcal{O}(\tau^3)\text{.}
	\end{split}
	\end{equation*}
	Since $\operatorname{Exp}_x (\tau v) = \gamma (\tau)$, this shows that
	\begin{equation*}
	\operatorname{p-Ret}_x(\tau v) = \operatorname{Exp}_x (\tau v) + \mathcal{O}(\tau^3)
	\end{equation*}
	as $\tau\rightarrow 0$, which proves that $\operatorname{p-Ret}$ is indeed a second-order retraction. 
\end{proof}

Another aspect of parameterization-based approaches concerns the computation of tangent vectors $\tilde{v}$ in the parameter domain that are \emph{uniformly distributed with respect to the Riemannian metric} $g$.
An efficient procedure for this task is presented in Algorithm \ref{alg:sampling}.

\begin{algorithm}
	\caption{Uniform sampling of $\tilde{v}$ w.r.t.~the Riemannian metric}\label{alg:sampling}
	\begin{algorithmic}[1]
		\State \textbf{Input:} Chart $(U, \phi)$; current position $\tilde{x}\in U \subset \mathbb{R}^m$
		\State Sample $m$ standard normal variates $w_1, \dots, w_m$.
		\State Compute
		\begin{equation}\label{eq:rand-dir-sphere}
		z := \left(\sum_{i=1}^m w_i^2\right)^{-\frac{1}{2}} (w_1, \dots, w_m)^T\ .
		\end{equation}
		%\State Compute the singular value decomposition
		%\begin{equation}\label{eq:rand-svd}
		%	(d\phi)\lvert_{\tilde{x}} = \mathbf{U} \mathbf{ \Sigma} \mathbf{V}^T \ .
		%\end{equation}
		\State Compute the singular value decomposition (SVD) of the pullback metric $(g_{ij}) = \phi^{*}g$ %= (d\phi)^T d\phi$
		at point $\tilde{x}$:
		\begin{equation*}%\label{eq:eigendec-metric}
		(g_{ij})\lvert_{\tilde{x}} = \mathbf{V}\mathbf{\Sigma} \mathbf{V}^T
		\end{equation*}
		with diagonal matrix $\mathbf{\Sigma}$ and orthogonal matrix $\mathbf{V}$.
		\State Compute
		\begin{equation*}%\label{eq:uniSampChart}
		\tilde{v} := \sum_{i=1}^m \frac{z_i}{\sqrt{\mathbf{\Sigma}_{ii}}} \mathbf{V}_i \ ,
		\end{equation*}
		where $\mathbf{V}_i$ the $i$th column vector of $\mathbf{V}$.
		\State \textbf{Return:} $\tilde{v}$
	\end{algorithmic}
\end{algorithm}

%	\begin{comment}
%	\begin{subtable}[h]{0.85\textwidth}
%		\centering
%	\begin{tabular}{ccccc} 
%		\toprule
%		\multirow{2}{*}{Dimension} & Symbolic & \multirow{2}{*}{Sampling} & Computing & Computing \\
%		& computations & & geodesics & retractions \\
%		\cmidrule{2-5}
%		$d=2$ & 1.653s & 0.109s & 7.038s & 0.054s (130x) \\ 
%		$d=5$ & 5.626s & 0.118s & 7.797s & 0.060s (130x) \\
%		$d=10$ & 50.349s & 0.134s & 8.911s & 0.069s (129x) \\
%		$d=15$ & & 0.150s & 10.525s & 0.086s (122x) \\
%		$d=20$ & & 0.172s & 12.073s & 0.099s (121x) \\
%		\bottomrule
%	\end{tabular}
%	\caption{Diagonal metric. Computed by using symbolic precomputations. \textbf{100\,000 steps each}.}
%	\end{subtable}
%
%	\vspace*{0.5cm}
%\end{comment}

Notice that Equation~\eqref{eq:rand-dir-sphere} is well defined since $w_k\neq 0$ a.s. for all $1\leq k \leq m$ and that $z$ is a sample of a uniformly distributed random variable on $\mathbb{S}^{m-1}$.
It is straightforward to verify the following claim:

\begin{lemma}
	The vector $\tilde{v}$ constructed in Algorithm \ref{alg:sampling}  is a unit tangent vector that is uniformly sampled with respect to the Riemannian metric $(g_{ij})$.
\end{lemma}

%\begin{comment}
%\begin{proof}
%	Consider a local parameterization $\phi: U\subset \mathbb{R}^m \to M \subset \mathbb{R}^n$ around $\tilde{x}\in U$, and let 
%	\begin{equation*}
%	(d\phi)\lvert_{\tilde{x}} = U D V^T
%	\end{equation*}
%	be the singular value decomposition of the differential $(d\phi)\lvert_{\tilde{x}}: \mathbb{R}^m \rightarrow \mathbb{R}^n$ (with orthogonal matrices $U\in \mathbb{R}^{n\times n}$, $V\in\mathbb{R}^{m\times m}$ and a rectangular diagonal matrix $D\in\mathbb{R}^{n\times m}$). Since the metric tensor takes the form $(g_{ij}) = (d\phi)^T d\phi$, the vector $\tilde{v}$ constructed in XX can be written as
%	\begin{equation*}
%	\tilde{v}= \sum_{i=1}^m \frac{z_i}{D_{ii}} V_i \ ,
%	\end{equation*}
%	where $\Sigma$ from XX satisfies $\Sigma = D^T D$.
%	Then the vector
%	\begin{equation*}
%	d\phi(\tilde{v}) = \sum_{i=1}^m z_i U_i 
%	\end{equation*} 
%	is a sample of a uniformly distributed random variable in the unit tangent space of $T_{\phi(\tilde{x})} M \subset \mathbb{R}^n$; thus indeed representing a sample of a uniformly distributed unit tangent vector with respect to the Riemannian metric $g$ on $M$ -- due to the \emph{isometric} embedding of $M$ into $\mathbb{R}^n$.	
%\end{proof}
%\end{comment}

\begin{figure*}[t]
	\centering
	\includegraphics[align=t,width=0.25\textwidth]{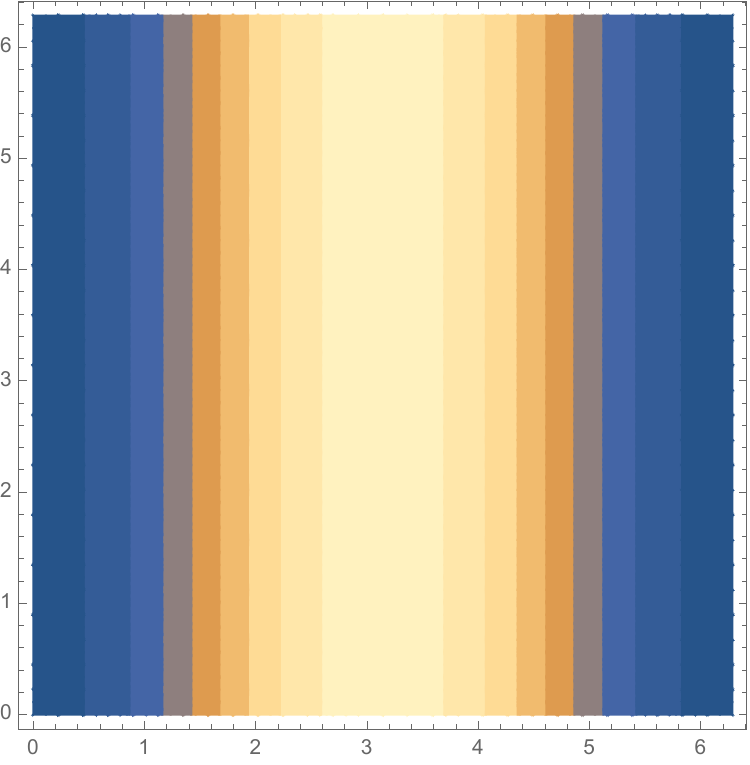}
	\includegraphics[align=t,width=0.05\textwidth]{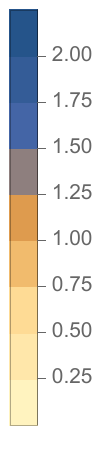}\hfill
	\includegraphics[align=t,width=0.25\textwidth]{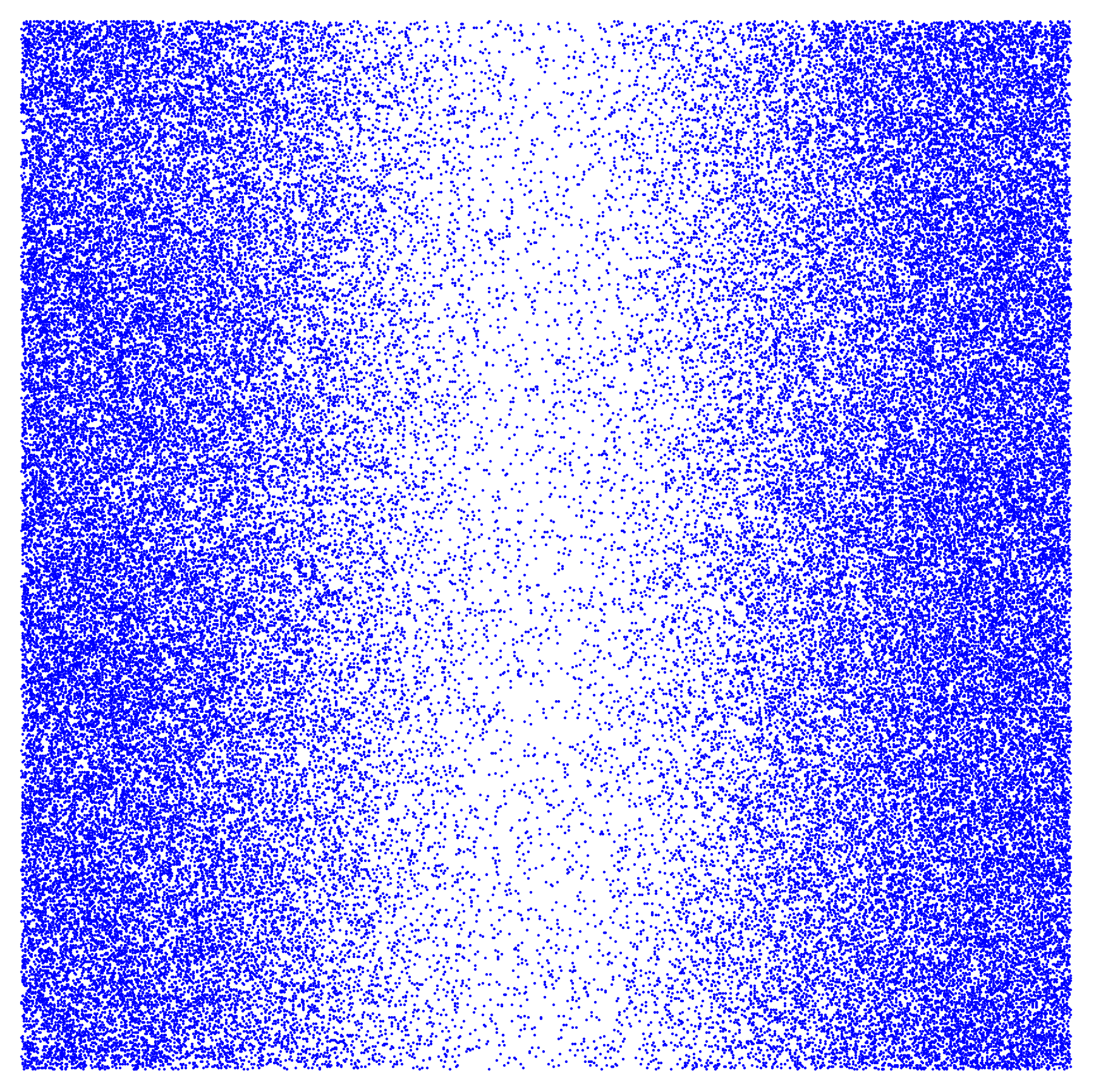}\hfill
	\includegraphics[align=t,width=0.4\textwidth]{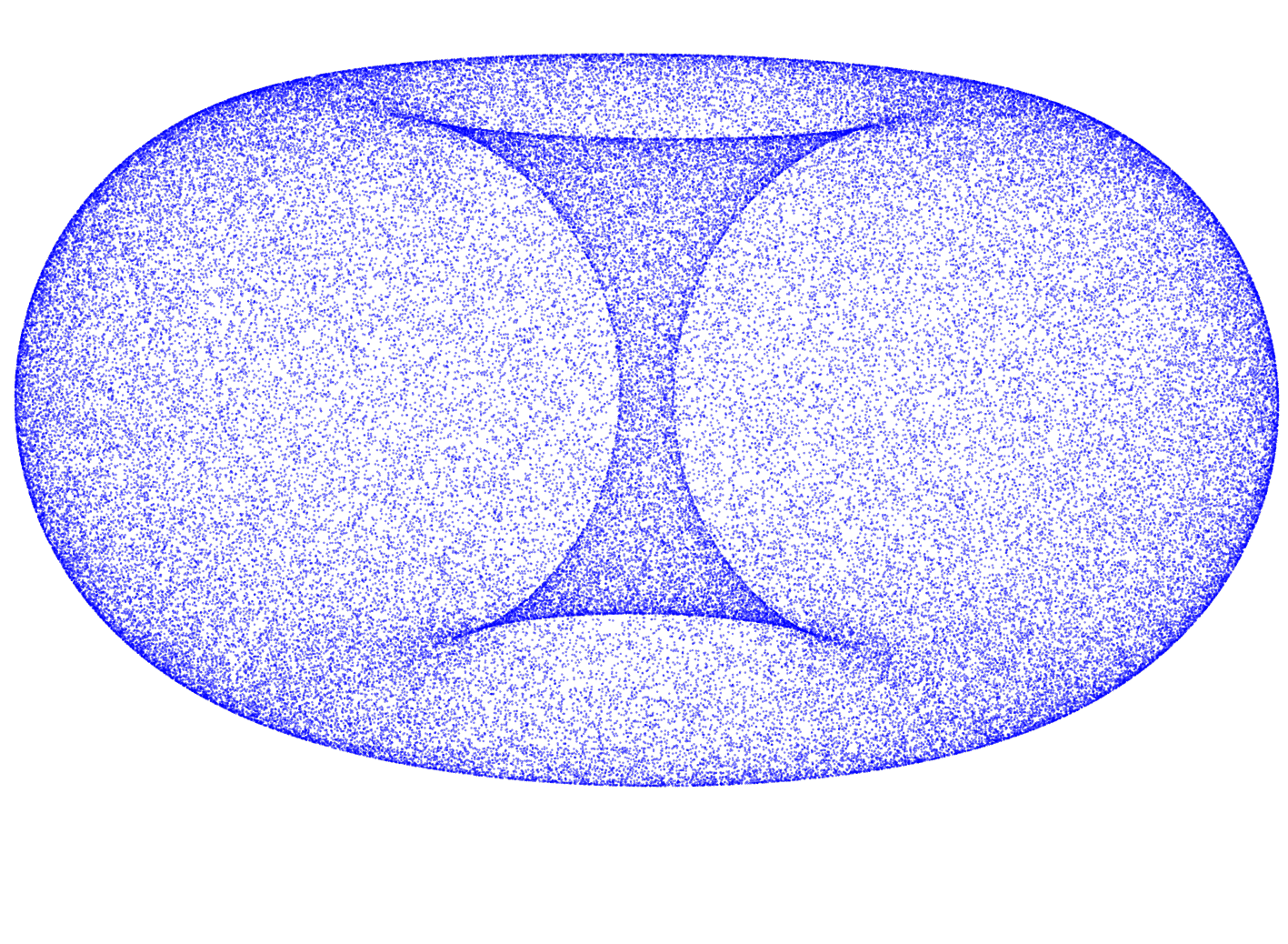}
	\caption{Results for a parametric surface defined as $\phi(s,t)= ((1.1 + \cos(s)) \cos(t), (1.1 + \cos(s)) \sin(t), \sin(s))$. \emph{Left:} Contour plot for the density of the stationary measure (i.e., the volume form of the Riemannian metric) in the parameter domain. \emph{Middle \& right:} Result of an approximate geodesic random walk based on $\operatorname{p-Ret}$ (using a single chart with periodic boundary conditions) with $100.000$ steps and stepsize $\varepsilon=0.5$, depicted in the parameter domain and in $\mathbb{R}^3$, respectively. Using an implementation in \textsc{Mathematica} on a standard laptop, the computation takes a few seconds. For a discussion of the relationship between the stationary measure and the empirical measure resulting from the random walk, see Section~\ref{ssec:inv-measure}.}
	\label{fig:param-walk-1}
\end{figure*}

In summary, using $\tilde{v}$ as computed in Algorithm \ref{alg:sampling}, setting $\bar v = \tilde{v}$ in Step 3 of Algorithm~\ref{alg:retraction}, and using
the retraction $\operatorname{p-Ret}$ from~\eqref{eq:ret-param}, provides an efficient method for computing random walks for parameterized manifolds. See Table \ref{tab:complexity} for a comparison between geodesic random walks and retraction-based random walks, showing that the latter are significantly faster than the former. Notice that the attendant algorithm might require to update the respective local chart during the simulated random walk.
Results of a simulated random walk on a 2-torus (using a single chart with periodic boundary conditions) are presented in Figure~\ref{fig:param-walk-1}.

\begin{table}
	\centering
	\begin{tabular}{cccc} 
		\toprule
		\multirow{2}{*}{Dimension} & {Sampling} & Computing & Computing \\
		& (SVD) & geodesics & retractions \\
		\cmidrule{3-4}
		$m=2$ & 0.107s %/0.083s 
		& 6.112s & 0.083s (74x) \\ 
		$m=5$ & 0.193s%/0.088s 
		& 13.151s & 0.190s (69x) \\
		$m=10$ & 0.322s%/0.207s 
		& 43.581s & 0.675s (65x) \\
		$m=15$ & 0.516s%/0.395s 
		& 125.081s & 1.994s (63x) \\
		$m=20$ & 0.873s%/0.570s 
		& 319.442s & 4.987s (64x) \\
		\bottomrule
	\end{tabular}
	
	\vspace{0.2cm}
	
	\caption{Computation times for geodesic and retraction-based random walks on $m$-dimensional tori with a non-diagonal metric $(g_{ij})=B^{T} (\bar g_{ij}) B$, where $\bar g_{ii} = 1.5 + \cos (x_{m-(i-1)})$ for $i = 1, \dots, m$, $\bar g_{ij} = 0$ for $i\neq j$, and $B$ is a random orthogonal matrix. In each case $10\,000$ steps were calculated in varying dimensions with stepsize $\varepsilon = 0.1$. All computations were performed on a standard laptop in \textsc{Matlab}, using  \textsc{Matlab}'s method ode45 for solving the geodesic equation. The Christoffel symbols were computed numerically. 
	}
	\label{tab:complexity}
\end{table}

\subsection{Projection-based retractions}\label{ssec:example2}
Our second example of computationally efficient retractions is based on a step-and-project method for compact manifolds that are given as zero level sets of smooth functions. More specifically, we consider the setting where 
\begin{equation*}
f: \mathbb{R}^n \to \mathbb{R}^k
\end{equation*}
is a smooth function, $k<n$, and the $m$-dimensional manifold in question, with $m=(n-k)$, is given by
\begin{equation}\label{eq:implicit-ret}
M = \{x \in \mathbb{R}^n \; : \; f(x) =0 \} \ .
\end{equation}
We assume that $0$ is a regular value of $f$, i.e, that the differential $df$ has full rank along $M$. In this case, a retraction can be defined via
\begin{equation*}
\operatorname{\pi-Ret}_x (v) = \pi_M (x+v) \ ,
\end{equation*}
where $x\in M$, $v\in T_xM$, and $ \pi_M(y) = \argmin_{p\in M} \|p-y\|$ denotes the (closest point) projection to $M$. Notice that $ \pi_M(y)$ is well defined as long as the distance $d(y,M)$ is less then the so-called \emph{reach} of $M$. The reach of $M$ is the distance of $M$ to its medial axis, which in turn is defined as the set of those points in $\mathbb{R}^n$ that do not have a unique closest point on $M$. Smoothly embedded compact manifolds always have positive reach. 

The following result is important in our setup. For a proof, see, e.g., Theorem 4.10 in~\cite{adler2002}. 
\begin{lemma}
	$\operatorname{\pi-Ret}$ is a second-order retraction.
\end{lemma}

It remains to specify the computation of a random unit direction $v\in T_xM$ as well as the implementation of the projection step. As for the former, a uniformly sampled unit tangent direction at some $x\in M$ can be computed by first sampling a randomly distributed unit vector $u \in \mathbb{S}^{n-1}$ in ambient space, and denoting its projection to $\mathrm{im} (df\lvert_{x})$ by $\tilde{u}$. Then 
\begin{equation}\label{eq:rand-dir-implicit}
\bar{v}= \frac{u-\tilde u}{\|u-\tilde u\|}
\end{equation}
yields the requisite uniformly distributed unit tangent vector. Finally, for computing the closest point projection of some given point $y\in \mathbb{R}^n$ to $M$, consider the Lagrangian 
\begin{equation*}
\mathcal{L}_y(z, \lambda) = \frac{1}{2}\lVert z- y\rVert^2 - \lambda^T f(z) \ ,
\end{equation*}
where $\lambda\in \mathbb{R}^{k}$ denotes the (vector-valued) Lagrange multiplier. Then Newton's method offers an efficient implementation for solving the Euler--Lagrange equations resulting from $\mathcal{L}_y$. In our implementation, we let Newton's method run until the threshold $\lvert f(z)\rvert/\|df\lvert_z\| < \varepsilon^3$ is satisfied. Notice that efficiency results from the fact that the point $y=x+\tau v$ is close to $M$ for $x\in M$ and $\tau$ small enough.

In summary, using $\bar{v}$ as computed in~\eqref{eq:rand-dir-implicit} for Step 3 of Algorithm~\ref{alg:retraction}, together with 
the retraction $\operatorname{\pi-Ret}$ from~\eqref{eq:implicit-ret}, provides an efficient method for computing random walks for implicitly given submanifolds. 
Results of a simulated random walk on the 2-torus are presented in Figure~\ref{fig:implicit-walk-1}.

\begin{remark}
	Our projection-based retraction is an instance of retractions induced by so-called \emph{normal foliations}. The latter are defined as foliations in a local neighborhood of $M= f^{-1}(0)$ whose leaves have dimension $k$ and intersect $M$ orthogonally. As shown in~\cite{zhang2020newton}, retractions induced by normal foliations are always of second order and therefore provide alternative approaches for sampling Brownian motion on implicit manifolds. Examples of retractions induced by normal foliations  include the gradient flow along $-\nabla f$ and Newton retractions based on the update step $\delta z = -(d f)^{\dagger}(z) f(z)$, where $(d f)^{\dagger}$ is the Moore--Penrose pseudo-inverse of the Jacobian $d f$.
\end{remark}

\begin{figure*}
	\centering
	\includegraphics[width=0.9\textwidth]{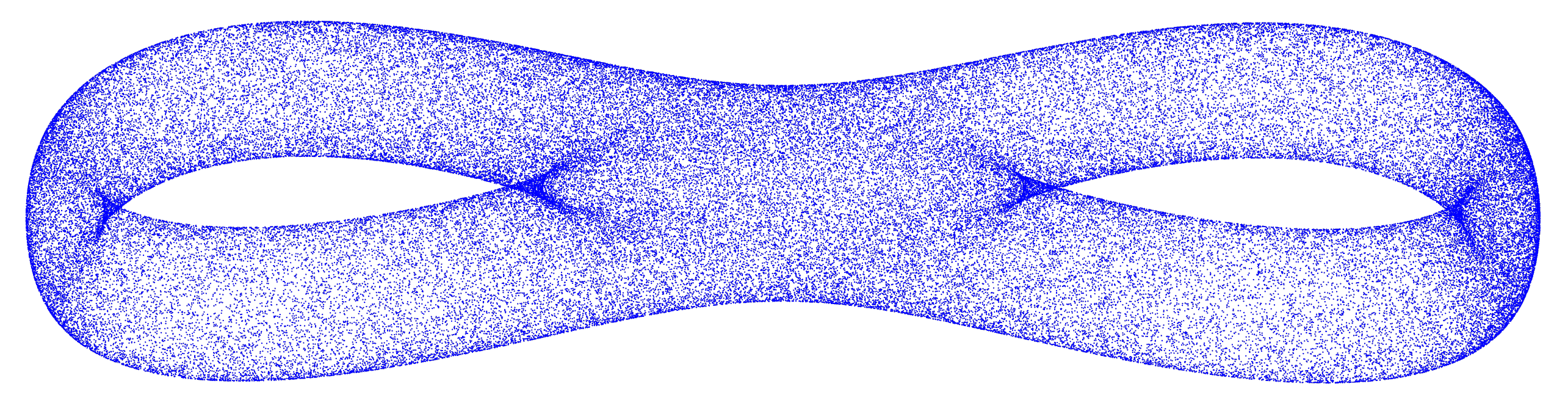}
	\caption{Result of a random walk using $\operatorname{\pi-Ret}$ for a genus two surface given as the zero-level set of $f:\mathbb{R}^3\to \mathbb{R}, \ f(x,y,z)=(x^2 (1 - x^2) - y^2)^2 + z^2 - 0.01$ with $100.000$ steps and stepsize $\varepsilon=0.1$. Using an implementation in \textsc{Mathematica} on a standard laptop, the computation takes less than a minute.}
	\label{fig:implicit-walk-1}
\end{figure*}

\subsection{Stationary measure}\label{ssec:inv-measure}
For second order retractions, we show below that the stationary measure %resulting from 
of the process in Algorithm \ref{alg:retraction} converges to the stationary measure of the Brownian motion, which is  the normalized Riemannian volume form, see Theorem~\ref{thm:inv-measure}. Nonetheless, the \emph{empirical} approximation $\frac{1}{N}\sum_{i=1}^N \delta_{x_i^\varepsilon}$ of the Riemannian volume form using random walks -- e.g., in order to obtain results as in Figs.~\ref{fig:param-walk-1} \&~\ref{fig:implicit-walk-1} -- requires the stepsize $\varepsilon$ and the number of steps $N$ in Algorithm \ref{alg:retraction} to be chosen appropriately. An indicator at which time the empirical approximation is close to the stationary measure is the so called $\delta$-cover time, i.e., the time needed to hit any $\delta$-ball. Hence, as a rough heuristic choice of $\varepsilon$ and $N$, one may utilize \cite{dembo2003,dembo2004}. For a Brownian motion on a $2$-dimensional manifold $M$, the $\delta$-cover time $C_\delta$ of $M$, satisfies (\cite{dembo2004}, Theorem 1.3)
\begin{equation*}
\lim_{\delta\rightarrow 0}\frac{C_\delta}{(\log\delta)^2} = \frac{2A}{\pi} \quad\text{a.s.}\ ,
\end{equation*}
where $A$ denotes the Riemannian area of $M$.
Based on the observation that a second-order retraction converges to Brownian motion in the limit of vanishing stepsize $\varepsilon$, and using that the physical time of the walk resulting from Algorithm~\ref{alg:retraction} is given by $\varepsilon^2 N$, one may set $\varepsilon^2 N= C_\delta$. If $\delta$ is small enough, one may then choose $N$ according to
\begin{equation*}
N \geq \frac{2A}{\pi} \left(\frac{\log \delta}{\varepsilon}\right)^2\ .
\end{equation*}
For dimension $m>2$, the corresponding result in \cite{dembo2003}, Theorem 1.1 yields the heuristic choice
\begin{equation*}
N \geq m\kappa_M \frac{-\delta^{2-m}\log\delta}{\varepsilon^2}\ ,
\end{equation*}
where
\begin{equation*}
\kappa_M = \frac{2}{(m-2)\omega_m} V(M)\ ,
\end{equation*}
$V(M)$ denotes the volume of $M$, and $\omega_m$ the volume of the unit sphere in $\mathbb{R}^m$.

Such heuristics, however, do not quantify the error between the \emph{empirical measure} obtained by simulating a random walk and the Riemannian volume measure.
In the setting of Euclidean random walks $(Y_n)_{n\in\mathbb{N}}$ with invariant distribution $\nu$, it has been shown in \cite{kloeckner2020} that under some conditions on $Y$ the expectation of the $1$-Wasserstein distance between the empirical and the invariant measure asymptotically decreases according to the law
\begin{equation}\label{eq:conv-wasserstein}
\mathbb{ E}\left\lVert \frac1n \sum_{i=1}^n \delta_{Y_i} - \nu \right\rVert_{W^1} \lesssim \alpha \frac{\left(\log (\vartheta n)\right)^{\beta}}{(\vartheta n)^\gamma}
\end{equation}
as $N\to \infty$ with constants $\alpha,\beta,\gamma > 0$ depending on the dimension $m$ and $\vartheta \in (0,1)$ depending on the process $Y$. Our simulations, see Figure \ref{fig:convergence-speed}, followed by a simple fit to the model given by the right hand side of Equation \eqref{eq:conv-wasserstein} suggest that a similar qualitative behavior might be present for retraction-based random walks on Riemannian manifolds. We leave a deeper analysis for future research.

\begin{figure}
	\centering
	\begin{subfigure}[b]{0.475\textwidth}
		\centering
		\includegraphics[width=\textwidth]{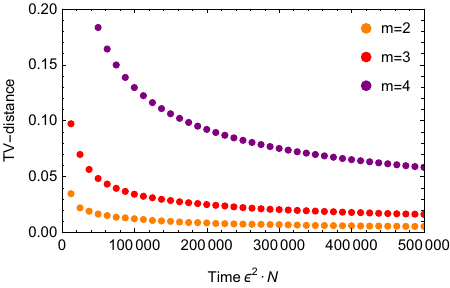}
		\caption{{\small Convergence speed in different dimensions with fixed stepsize $\varepsilon=0.25$. %Parameters in the model fit: $d=2$: $a = 0.0705$, $b = 0.3204$, $c = 0.5176$; $d=3$: $a = 0.0705$, $b = 0.3204$, $c = 0.5176$; $d=4$: $a = 0.0705$, $b = 0.3204$, $c = 0.5176$.
		}}    
	\end{subfigure}
	\hfill
	%\begin{subfigure}[b]{0.475\textwidth}  
	%	\centering 
	%	\includegraphics[width=\textwidth]{TV3.pdf}
	%	\caption{{\small $d=3$. Parameters in the model fit: $a = 0.0135$, $b = 0.9597$, $c = 0.7336$ and $d = 0.1153$.}}     
	%\end{subfigure}
	%\vskip\baselineskip
	\begin{subfigure}[b]{0.475\textwidth}   
		\centering 
		\includegraphics[width=\textwidth]{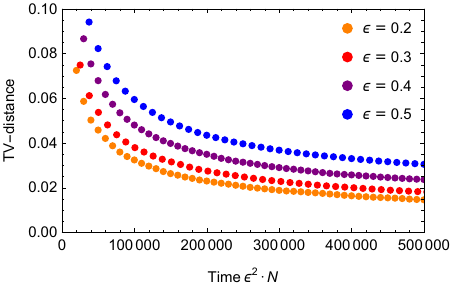}
		\caption{{\small Convergence speed of walks with different stepsizes in dimension $m=3$. %Parameters in the model fit: $a = 0.0111$, $b = 1.0208$, $c = 0.6090$.
		}}    
	\end{subfigure}
	%\hfill
	%\begin{subfigure}[b]{0.475\textwidth}   
	%	\centering
	%	\includegraphics[width=\textwidth]{TV5.pdf}
	%	\caption{{\small $d=5$. Parameters in the model fit: $a = 0.0103$, $b = 1.0318$, $c = 0.5461$ and $d = 0.1838$.}}   
	%\end{subfigure}
	\caption{Convergence speed of retraction-based random walks to the invariant measures on $m$-dimensional tori parametrized by $[0,2\pi]^m$ with periodic boundary conditions and equipped with the diagonal metric $g_{ii} = 1.5 + \cos (x_{m-(i-1)})$ for $i = 1, \dots , m$ and $g_{ij} = 0$ for $i\neq j$. Results are shown for varying dimension $m$ and fixed stepsize $\varepsilon$ (\emph{left}) as well as fixed dimension $m$ and varying stepsize $\varepsilon$ (\emph{right}). 
		Simulations were performed $10$ times for each combination of dimension and stepsize. The points display the arithmetic mean of the total variation distances between the invariant measure and the empirical measures resulting from a box count using $20^m$ equally sized cubes.
	}
	\label{fig:convergence-speed}
\end{figure}
%We leave these questions for future work. 

\section{Convergence}\label{sec:convergence}
This section is devoted to the proof of Theorem \ref{thm:general-convergence}, which yields convergence in the Skorokhod topology of the process constructed in Algorithm \ref{alg:retraction} as the stepsize $\varepsilon\rightarrow 0$.
The Skorokhod space $\mathcal{D}_M[0,\infty)$ is the space of all càdlàg, i.e., right-continuous functions with left-hand limits from $[0,\infty)$ to $M$, equipped with a metric turning $\mathcal{D}_M[0,\infty)$ into a complete and separable metric space, and hence a Polish space (see e.g. \cite{ethier1986}, Chapter 3.5). The topology induced by this metric is the Skorokhod topology, which generalizes the topology of local uniform convergence for continuous functions $[0,\infty)\rightarrow M$ in a natural way. For local uniform convergence we have uniform convergence on any compact time interval $K\subset [0,\infty)$. In the Skorokhod topology functions with discontinuities converge if the times and magnitudes of the discontinuities converge (in addition to local uniform convergence of the continuous parts of the functions). Thus,  
$\mathcal{D}_M[0,\infty)$ is a natural and -- due to its properties -- also a convenient space for convergence of discontinuous stochastic processes.

With reference to the notation used in Theorem~\ref{thm:general-convergence}, notice that the process $X^\varepsilon = \big(X_{\lfloor\varepsilon^{-2}t\rfloor}^\varepsilon\big)_{t\geq 0}$ has deterministic jump times and is therefore not Markov as a process in continuous time. In order to apply the theory of Markov processes, let $\left(\eta_t\right)_{t\geq 0}$ be a Poisson process with rate $\varepsilon^{-2}$ and define the pseudo-continuous process $Z^\varepsilon = (Z_t^\varepsilon)_{t\geq 0}$ by $Z_t^\varepsilon := X_{\eta_t}^\varepsilon$ for any $t\geq 0$, which is Markov since the Poisson process is Markov. Then, the convergence of $Z^\varepsilon$ implies convergence of $X^\varepsilon$ to the same limit by the law of large numbers, see \cite{kallenberg2002}, Theorem~19.28. First we restate Theorem \ref{thm:general-convergence}, using an equivalent formulation based on the process $Z^\varepsilon$:

\begin{theorem}\label{thm:convergence}
	Consider the sequence of random variables $(X_i^\varepsilon)_{i\in\mathbb{N}}$ obtained by the construction from Algorithm \ref{alg:retraction}. The process $Z^\varepsilon = \left( X^\varepsilon_{\eta_t}\right)_{t\geq 0}$ converges in distribution in the Skorokhod topology to the $L$-diffusion $Z$, i.e., the stochastic process with generator $L$ defined in \eqref{eq:generator}. This means that for any continuous and bounded function 
	$f:\mathcal{D}_M[0,\infty) \rightarrow \mathbb{ R}$ we have
	$$\mathbb{E}(f((Z_t^\varepsilon)_{t\geq 0})) \rightarrow \mathbb{E}(f((Z_t)_{t\geq 0})) \ , \quad \varepsilon \rightarrow 0 \ .$$
\end{theorem}

\noindent
Certainly one cannot, in general, deduce convergence of stationary measures from time-local convergence statements about paths. However, standard arguments from the theory of Feller processes indeed allow to do so. This yields the following statement regarding stationary measures.

\begin{theorem}\label{thm:inv-measure}
	Let $\mu_\varepsilon$ be the stationary measure of the retraction-based random walk with stepsize $\varepsilon$. Then the weak limit of $\left(\mu_\varepsilon\right)_{\varepsilon>0}$ as $\varepsilon\rightarrow 0$ is the stationary measure of the $L$-diffusion, with generator $L$ defined in \eqref{eq:generator}. In the case of second-order retractions, the limit is the Riemannian volume measure.
\end{theorem}

\begin{remark}
	Notice that due to the approximation of the exponential map, the process $Z^\varepsilon$ is \emph{not in general reversible} and the generator of $Z^\varepsilon$ is not in general self-adjoint. This is relevant since the stationary measure corresponds to the kernel of the \emph{adjoint} of the generator.
\end{remark}

Our proof of Theorem~\ref{thm:convergence} and Theorem~\ref{thm:inv-measure} (presented below) hinges on convergence of generators that describe the infinitesimal evolution of Markov processes, see, e.g., \cite{ethier1986}, Chapter 4. We show that the generator of $Z^\varepsilon$ converges to the generator $L$ defined in \eqref{eq:generator}.
The generator of the process $Z^\varepsilon$ is spelled out in Lemma \ref{lem:ret-generator}, and the convergence of this generator to $L$ is treated in Lemma \ref{lem:generator-convergence}.

The following result is standard for transition kernels $U^\varepsilon$ on compact state spaces, see, e.g., \cite{ethier1986}, Chapter 8.3:

\begin{lemma}\label{lem:ret-generator}
	The process $(Z_t^\varepsilon)_{t\geq 0}$ is Feller, and its generator is given by
	\begin{equation}
	L^\varepsilon f = \frac{1}{\varepsilon^2}(U^\varepsilon f - f)
	\end{equation}
	for all $f\in C(M)$, the continuous real valued functions on $M$. Here, $$(U^\varepsilon f)(x) =  \frac{1}{\omega_m}\int_{S_x} f(\operatorname{Ret}_x (\varepsilon \sqrt{m}\, \bar{v})) \text{d}\bar{v}$$ and $\omega_m$ is the volume of the unit $(m-1)$-sphere.
\end{lemma}

Notice that Lemma \ref{lem:ret-generator} additionally states that the process $Z^\varepsilon$ satisfies the Feller property, i.e., that its semigroup is a contraction semigroup that also fulfills a right-continuity property. Therefore, in particular, the process $Z^\varepsilon$ is Markov.

In the following, $C^{(2,\alpha)}(M)$, $0<\alpha\leq 1$ denotes the space of two times differentiable functions on $M$ with $\alpha$-Hölder continuous second derivative.

\begin{lemma}\label{lem:generator-convergence}
	For any $f\in C^{(2,\alpha)}(M)$ with $0 < \alpha \leq 1$,
	\begin{equation*}
	\left\lVert L^\varepsilon f - L f\right\rVert_\infty \rightarrow 0
	\end{equation*}
	as $\varepsilon\rightarrow 0$, where $L^\varepsilon$ is the operator from Lemma \ref{lem:ret-generator} and $L$ as defined in \eqref{eq:generator}.
\end{lemma}

\begin{proof}
	Let $P^\varepsilon$ be the transition kernel of the geodesic random walk, i.e.,
	\begin{equation*} 
	(P^\varepsilon f)(x) = \frac{1}{\omega_m}\int_{S_x}f(\operatorname{Exp}_x(\varepsilon \sqrt{m}\, \bar{v}))\text{d}\bar{v} \ .
	\end{equation*}
	For $x\in M$ and $v\in T_x M$ define
	\begin{equation*}
	(Gf)(x,v) := 
	\frac{1}{2} df\big\lvert_x \left(\frac{D}{d\tau} \frac{d}{d\tau} \operatorname{Ret}_x (\tau v)\right)\Bigg\lvert_{\tau = 0} 
	\end{equation*}
	and
	\begin{equation*}
	(\tilde Gf)(x) := 
	\frac{m}{\omega_m}\int_{S_x} (Gf)(x,\bar{v})\text{d}\bar{v} \ .
	\end{equation*}
	Then, since $Lf =\frac12 \Delta_g f+ \tilde Gf$,
	\begin{align*}
	\left\lVert L^\varepsilon f - Lf\right\rVert_\infty \leq \left\lVert \frac{1}{\varepsilon^2}(U^\varepsilon f - P^\varepsilon f)-\tilde{G}f \right\rVert_\infty 
	+ \left\lVert \frac{1}{\varepsilon^2} (P^\varepsilon f - f) - \frac12 \Delta_g f\right\rVert_\infty\text{.}
	\end{align*}
	Proposition 2.3 in \cite{jorgensen1975} proves that the second summand on the right hand side converges to $0$ as $\varepsilon\rightarrow 0$. Let $\tilde\varepsilon := \sqrt{m}\,\varepsilon$. Then
	\begin{align*}
	&\left\lvert\frac{1}{\varepsilon^2}(U^\varepsilon f(x) - P^\varepsilon f(x)) - (\tilde{G}f)(x)\right\rvert \\
	\leq &\frac{m}{\omega_m}\int_{S_x} \Big\lvert \frac{1}{\tilde\varepsilon^2}[f(\operatorname{Ret}_x (\tilde\varepsilon \bar{v})) -f(\operatorname{Exp}_x (\tilde\varepsilon \bar{v}))] - (Gf)(x,\bar{v})\Big\rvert \text{d}\bar{v}\ .
	\end{align*}
	Now consider the Taylor expansions of the functions $f\circ \operatorname{Ret}_x$ and $f\circ \operatorname{Exp}_x$ in $\tilde\varepsilon$ at $0$. To this end, notice that for any $x\in M$ and $v\in T_xM$, one has
	\begin{align*}
	\frac{d}{d \tau} f(\operatorname{Ret}_x (\tau v)) 
	= df\lvert_{\operatorname{Ret}_x(\tau v)} (v_\tau ) \ , 
	\end{align*}
	where $v_\tau = \frac{d}{d\tau} \operatorname{Ret}_x (\tau v)$.
	Hence, we obtain that
	\begin{align*}
	R_{\operatorname{Ret}}^{x,v} (\tau) &:=\frac{d^2}{d \tau^2} f(\operatorname{Ret}_x (\tau v)) 
	= \operatorname{Hess} f(v_\tau, v_\tau) + df\lvert_{\operatorname{Ret}_x(\tau v)} \left(\frac{D}{d\tau} v_\tau\right) \text{,}
	\end{align*}
	where $\operatorname{Hess} f = \nabla d f$ is the Riemannian Hessian of $f$.
	Using that $v_0 = v$, the desired Taylor expansion  takes the form
	\begin{equation*}
	f(\operatorname{Ret}_x (\tilde\varepsilon v)) = f(x) + \tilde\varepsilon df\lvert_x(v)
	+ \frac{1}{2}\tilde\varepsilon^2 R_{\operatorname{Ret}}^{x,v} (\xi_1)\ ,
	\end{equation*}
	where $\xi_1\in [0,\tilde\varepsilon]$.
	Defining $u_\tau = \frac{d}{d\tau} \operatorname{Exp}_{x}(\tau v)$, %u_\tau = d\operatorname{Exp}_x\lvert_{\tau v} (v)
	we similarly obtain that
	\begin{equation*}
	\begin{split}
	f(\operatorname{Exp}_x (\tilde\varepsilon v)) &= f(x) + \tilde\varepsilon df\lvert_x(v) + \frac{1}{2}\tilde\varepsilon^2 R_{\operatorname{Exp}}^{x,v} (\xi_2) \ ,
	\end{split}
	\end{equation*}
	with $\xi_2 \in [0,\tilde\varepsilon]$ and
	\begin{equation*}
	R_{\operatorname{Exp}}^{x,v}(\tau) = \operatorname{Hess} f(u_{\tau}, u_{\tau})\ ,
	\end{equation*}
	where for the last equality we used the fact that 
	$
	\frac{D}{d\tau} u_{\tau} = 0.
	$
	Since $R_{\operatorname{Ret}}^{x,v} (0) = R_{\operatorname{Exp}}^{x,v} (0) + 2(Gf)(x,v)$,
	we obtain that
	\begin{equation*}
	\begin{split}
	&\lvert f(\operatorname{Ret}_x (\tilde\varepsilon\, v)) - f(\operatorname{Exp}_x (\tilde\varepsilon\, v))- \tilde\varepsilon^2 (Gf)(x,v)\rvert \\
	=& \frac{\tilde\varepsilon^2}{2}\lvert R_{\operatorname{Ret}}^{x,v}(\xi_1 ) - R_{\operatorname{Ret}}^{x,v}(0) - R_{\operatorname{Exp}}^{x,v}(\xi_2) + R_{\operatorname{Exp}}^{x,v}(0)\rvert \\
	\leq & \frac{\tilde\varepsilon^2}{2}\Big(\lvert R_{\operatorname{Ret}}^{x,v}(\xi_1 ) - R_{\operatorname{Ret}}^{x,v}(0)\rvert + \lvert R_{\operatorname{Exp}}^{x,v}(\xi_2) - R_{\operatorname{Exp}}^{x,v}(0)\rvert\Big)\text{.}
	\end{split}
	\end{equation*}
	Since $\operatorname{Hess} f$ is Hölder-continuous by assumption and $\operatorname{Exp}$ and $\operatorname{Ret}$ are smooth, both $R_{\operatorname{Exp}}$ and $R_{\operatorname{Ret}}$ are Hölder-continuous as products and sums of Hölder-continuous functions. Therefore, there exists a constant $C>0$ such that one has the uniform bound
	\begin{align*}
	\lvert f(\operatorname{Ret}_x (\tilde\varepsilon\, v)) - f(\operatorname{Exp}_x (\tilde\varepsilon\, v))-\tilde\varepsilon^2 (Gf)(x,v)\rvert 
	\leq \frac{C}{2} ( \xi_1^\alpha + \xi_2^\alpha)   \tilde\varepsilon^2 \leq C \tilde \varepsilon^{2+\alpha}  \ .
	\end{align*}
	Hence,
	\begin{align*}
	\sup_{x\in M}\left\lvert \frac{1}{\varepsilon^2} \left(U^\varepsilon f(x) - P^\varepsilon f(x)\right)-(\tilde{G}f)(x)\right\rvert 
	\leq C m^{1+\frac{\alpha}{2}} \varepsilon^\alpha \xrightarrow{\varepsilon\rightarrow 0} 0 \ ,
	\end{align*}
	which proves the lemma.
\end{proof}

Using the above results, we can finish the proof of Theorem \ref{thm:convergence}.

\begin{proof}[Proof of Theorem \ref{thm:convergence}]
	By the proof of Theorem 2.1 and the remark on page 38 in \cite{jorgensen1975}, the space $C^{(2,\alpha)}(M)$ is a core for the differential operator $L$. 
	For a sequence of Feller processes $(Z^\varepsilon)_{\varepsilon > 0}$ with generators $L^\varepsilon, \varepsilon>0$ and another Feller process $Z$ with generator $L$ and core $D$ the following are equivalent (\cite{kallenberg2002}, Theorem 19.25):
	\begin{enumerate}
		\item[(i)] If $f\in D$, there exists a sequence of $f_\varepsilon\in \operatorname{Dom}(L^\varepsilon)$ such that $\lVert f_\varepsilon-f\rVert_\infty \rightarrow 0$ and $\left\lVert L^\varepsilon f_\varepsilon-Lf\right\rVert_\infty \rightarrow 0$ as $\varepsilon \rightarrow 0$.
		\item[(ii)] If the initial conditions satisfy $Z^\varepsilon_0 \rightarrow Z_0$ in distribution as $\varepsilon\rightarrow 0$ in $M$, then $Z^\varepsilon \rightarrow Z$ as $\varepsilon\rightarrow 0$ in distribution in the Skorokhod space $\mathcal{D}_M([0,\infty))$.
	\end{enumerate}
	In our case $D=C^{(2,\alpha)}(M)$ and $f_\varepsilon = f\in D$ satisfy (i) due to Lemma \ref{lem:generator-convergence}. This concludes the proof of Theorem~\ref{thm:convergence} since the $L$-diffusion is also a Feller process on a compact manifold.
	
\end{proof}
From an analytical point of view it is perhaps not surprising that convergence of generators, see Lemma \ref{lem:generator-convergence}, implies convergence of stationary measures as stated in Theorem \ref{thm:inv-measure}. For the proof we here follow standard arguments of the theory of Feller processes.

\begin{proof}[Proof of Theorem \ref{thm:inv-measure}]
	The family of (unique) stationary measures $\left(\mu_\varepsilon\right)_{\varepsilon> 0}$ is a tight family of measures since all measures are supported on the same compact manifold $M$. By Prokhorov's theorem, which provides equivalence of tightness and relative compactness, any subsequence of the family $(\mu_{\varepsilon_n})_{n\in\mathbb{N}}$ with $\varepsilon_n \rightarrow 0$ as $n\rightarrow\infty$ has a convergent subsubsequence. But the uniform convergence of the generators that we showed in Lemma \ref{lem:generator-convergence}, see also (i) in the proof of Theorem \ref{thm:convergence}, is also equivalent to the convergence of the associated semigroups with respect to the supremum norm in space, see again \cite{kallenberg2002}, Theorem 19.25. This implies (\cite{ethier1986}, Chapter 4, Theorem 9.10) that
	all subsequential limits must be the unique stationary measure of the $L$-diffusion, and therefore all subsubsequences converge to the same measure. A standard subsubsequence argument then proves the theorem. 
\end{proof}

\subsection*{Retractions on sub-Riemannian manifolds}
A possible extension of our work concerns random walks in \emph{sub-Riemannian} geometries. Indeed, sub-Riemannian structures can be used to model low-dimensional noise that lives in a higher dimensional state space~\cite{habermann2017}. A sub-Riemannian structure on a smooth manifold $M$ consists of a smoothly varying positive definite quadratic form on a sub-bundle $E$ of the tangent bundle $TM$. Similar to the Riemannian setting, so called \emph{normal} sub-Riemannian geodesics arise from a first order ODE on the cotangent bundle~\cite{montgomery2002}. Such geodesics are uniquely determined by an initial position $x\in M$ and a $1$-form $\alpha \in T^{*}_{x}M$, and can therefore be approximated by a second order Taylor expansion of the respective ODE on the cotangent bundle. Such an approximation would algorithmically resemble our retraction-based  approach, when attempting to efficiently simulate a sub-Riemannian geodesic random walk.  However, in general, there exists no canonical choice for \emph{sampling} a requisite $1$-form $\alpha$ in every step in order to construct a sub-Riemannian geodesic random walk (or an approximation thereof). Indeed, one requires an additional choice of a measure for such sampling, and different choices may lead to different limiting processes, see~\cite{agrachev2018, boscain2017}. Notice that once such a choice has been made, one can seamlessly adapt our algorithm to the sub-Riemannian setting. We leave for future work a respective convergence analysis. Moreover, there might exist canonical choices of such measures for special sub-Riemannian structures -- e.g., when considering the frame bundle of a smooth manifold $M$ in order to model anisotropic Brownian motion and general diffusion processes~\cite{grong2022, sommer2017}. 

\printbibliography

\end{document}